\documentclass[12pt]{amsart}
\usepackage{amscd}
\usepackage{amsmath}
\usepackage{latexsym}
\usepackage{graphicx}
\usepackage{amsfonts}
\usepackage{amssymb}
\usepackage{color}
\usepackage[all]{xy}

\newcommand{\R}{\mathbb{R}}

\newcommand{\Z}{\mathbb{Z}}

\DeclareMathOperator{\PH}{PH}
\newcommand{\Str}{\mathbf{Str}}
\newcommand{\data}{\mathbf{data}}
\newcommand{\per}{\mathbf{Per}}
\newcommand{\branch}{\mathbf{Br}}
\newcommand{\branchA}{\mathbf{Br}_\alpha}
\newcommand{\branchB}{\mathbf{Br}_\beta}
\newcommand{\branchC}{\mathbf{Br}_\gamma}
\newcommand{\branchAB}{\mathbf{Br}_{\alpha\beta}}
\newcommand{\branchAC}{\mathbf{Br}_{\alpha\gamma}}
\newcommand{\branchBC}{\mathbf{Br}_{\beta\gamma}}

\numberwithin{equation}{section}
\theoremstyle{plain}
\newtheorem{theorem}[equation]{Theorem}

\newtheorem{corollary}[equation]{Corollary}
\newtheorem{proposition}[equation]{Proposition}
\newtheorem{lemma}[equation]{Lemma}
\newtheorem{substuff}{\bf Remark}[equation] 

\theoremstyle{definition}

\newtheorem{definition}[equation]{Definition}

\newtheorem{example}[equation]{Example}

\newtheorem{construction}[equation]{Construction}

\theoremstyle{remark}

\newtheorem{subremark}[substuff]{Remark} 

\begin{document}
\title{Persistent homology with non-contractible preimages}
\date{\today}

\author{Konstantin Mischaikow}
\address{Math.\ Dept., Rutgers University, New Brunswick, NJ 08901, USA}
\email{mischaik@math.rutgers.edu}
\thanks{Mischaikow  was 
partially supported by NSF grants  1521771, 1622401, 1839294, 1841324, 1934924, by NIH-1R01GM126555-01 as part of the Joint DMS/NIGMS Initiative to Support Research at the Interface of the Biological and Mathematical Science, and a grant from the Simons Foundation.}

\author{Charles Weibel}
\address{Math.\ Dept., Rutgers University, New Brunswick, NJ 08901, USA}
\email{weibel@math.rutgers.edu}\urladdr{http://math.rutgers.edu/~weibel}
\thanks{Weibel was supported by NSF grants, and the Simonyi Endowment at IAS}

\begin{abstract} 
For a fixed $N$, we analyze the space of all sequences 
$z=(z_1,\dots,z_N)$, approximating a continuous function on the circle,
with a given persistence diagram $P$, and show that 
the typical components of this space are homotopy equivalent to $S^1$.
We also consider the space of functions on $Y$-shaped
(resp., star-shaped) trees with a 2-point persistence diagram, and show that this space is homotopy equivalent to $S^1$
(resp., to a bouquet of circles).
\end{abstract} 
\maketitle

\section*{Introduction}
Topological Data Analysis (TDA) is a rapidly developing body of techniques for 
the analysis of high dimensional data associated with nonlinear structures.
Persistent homology  has become one of the primary tools in TDA, for reasons
including efficiency of computation, robustness with respect to perturbations in the data, and dramatic data compression.
The focus of this paper is on understanding the loss of information due to this compression.

To the best of our knowledge, all applications of persistent homology 
to experimental data can be characterized as follows.
There exists a finite simplicial complex $K$ and a 
function $\varphi: K\to \R$ such that 
each sublevel set $K_r=\varphi^{-1}((-\infty,r])$ defines a subcomplex of $K$.
With this input, the persistent homology algorithm outputs  persistence diagrams $P = \{P_i\}$, 
one for each dimension of homology; we will 
view them as point clouds in $\R^2$.

With regard to compression, an obvious question is: given a fixed complex $K$ 
and persistence diagrams $\{P_i\}$, 
what is the "preimage" space of functions $\varphi$ that produce these diagrams? 
A detailed exposition of the polyhedral structure of these 
preimages is given in \cite{LT}. 
While this geometric structure is clearly valuable, it does not necessarily translate into an understanding of 
what information is lost due to compression. 

We will  consider the slightly more restrictive setting in which $\varphi:K\to\R$ is completely determined at $\sigma$ by its values 
on the set of adjacent vertices of $K$ by 
the formula $\varphi(\sigma) = \max_{v\prec \sigma}\varphi(v)$.
If $K$ has $N$ vertices, then $\varphi$ can be characterized by the vector 
$z=(z_1,\ldots, z_N)\in\R^N$.
Therefore, for fixed $K$ we can view persistent homology as inducing a function
\[
PH\colon \R^N \to \per,
\]
where $\per$ denotes the space of persistence diagrams. 
Given $P\in \per$, our long term goal is to understand the preimage 
$\data_P=PH^{-1}(P)$ as a subset of $\R^N$.

In \cite{CMW}, 
we considered the case where $K$ is a simplicial complex representing an interval, 
i.e.,\ $K$ consists of $N$ vertices, $N-1$ edges, and each vertex is the boundary of at most two edges.
The primary result is that 
each component of $\data_P$ 
is contractible (for the number of components see \cite{curry}).
The work in \cite{CMW} 
was motivated in part to apply topological fixed point theorems to 
nonlinear dynamics tracked via persistence diagrams.
In this setting, contractibility is a sufficiently strong condition, and thus, 
the collapse of geometry  does not result in a loss of  information.
We hasten to remark that, even in this simple setting, 
the geometry of $\data_P$ 
is nontrivial (see \cite{LT}). 
However, the machinery 
of \cite{CMW} that determines the topology of the preimage is fairly simple and thus worth pursuing in and of itself.

We consider two families of examples in this paper.
The first is where $K$ is an $N$-gon, representing a circle, 
and the second is where $K$ is a simple star-like tree, 
with at least three branches.
In the first case, the typical components  of $\data_P$ 
are homotopic to a circle, reflecting the topology of $K$;
see Corollary~\ref{cor:circle}.
The reader might find it surprising that in the second case,
when $K$ is a tree with only 3 branches, each of length at least two, then the preimage
$\data_P$ is homotopic to a circle;
adding more branches yields
a preimage which is homotopic to a bouquet of many circles 
(see Sections~\ref{sec:Y}--\ref{sec:starlike}).

Though based on the same machinery, the details of the proofs of these two examples are largely independent. 
As a result it is easy to see that for an arbitrary graph $K$  the topology of $\data_P$
can be quite complicated.
In fact, we limit ourselves to presenting these two examples precisely because we do not have a general description of the homotopy type of $\data_P$ 
based on $K$ and $P$. 

This paper is organized into two parts.  
The first part of the paper considers the circle, modelled by the
$N$-gon. In Sections \ref{sec:components} and \ref{sec:Str},
we discuss critical value sequences
and the poset $\Str$ of cellular strings. 
In Section \ref{sec:polytopes}, we show that the connected components
of a typical component are homeomorphic to the geometric realization
of $\Str$, and
analyze the homotopy type of various sub-posets of $\Str$
in Section \ref{sec:homotopy}, via a series of simple moves.
Finally, we prove the
main result in Section \ref{sec:main}.

The second family of examples is studied in Section \ref{sec:Y} (3 branches) and 
Section \ref{sec:starlike} (many branches).
The nontrivial topology in the preimage $\data_P$ arises 
from  the fact that there is a larger family of moves.

\section{Preliminaries}
\label{sec:prel}

As indicated in the introduction,
we consider a 1-dimensional complex $K$ and
fix a labelling of the vertices by  $i=1,\ldots, N$.
We consider functions $\varphi \colon K\to \R$ determined by the values of $\varphi$ on the vertices, 
i.e.,\ $\varphi(\sigma) = \max_{i\prec \sigma} \varphi(i)$.
This allows us to represent $\varphi$ as a vector $z=(z_1,...,z_{N})\in \R^N$, where $z_i=\varphi(i)$.

Each vector $z\in\R^N$ defines a sublevel-set filtration of $K$ as follows. For $r\in\R$, let $K_r(z)$ denote the subcomplex of $K$ 
whose vertices $i$ have $z_i\le r$ and whose edges $[i,j]$ satisfy 
$\max\{z_i,z_j\} \le r$. 
As $r$ varies, the homology $H_n$ of the $K_r(z)$ determines a persistence diagram $\PH_n(z)$.

We focus on $\PH_0$ in this paper, because we restrict our attention to complexes
that are either an $N$-gon or a tree. In these cases
$\PH_n(z)$ is empty for $n>1$, and $\PH_1(z)$ consists of
either the single point 
$(\max\{z_i\},\infty)$ or the empty set, respectively.
We write $M$ for the number of points in $PH_0(z)$.

A persistence diagram $P$ is considered {\it typical} if 
the coordinates of its $M$ points are distinct, and
we say that $z\in\R^N$ is a {\it typical point} if 
its coordinates are distinct.
Clearly, typical points have a typical persistence diagram.
We leave the proof of the following lemma to the reader.

\begin{lemma}\label{lem:tpz'}
Given a complex $K$ of the type discussed in this paper and an associated  
typical persistence diagram $PH_0(z)=\{ (p_m^b,p_m^d)\}_{m=1}^M$, 
then there is a typical point $z'$ with $PH_0(z)=PH_0(z')$.
\end{lemma}

\begin{definition}\label{def:cvs}
 If $z$ is a typical point,
we say that a coordinate $z_i$ is a {\it local maximum} 
(resp., a {\it local minimum}) if 
$z_i\ge z_j$ (resp., $z_i\le z_j$)
for all vertices $[j]$ adjacent to $[i]$.
The vector $(z_{i_1},...,z_{i_{L}})$ in $\R^L$ of 
the local maxima and minima in $z$ is the 
{\it critical value sequence} of the typical point $z\in\R^N$.
\end{definition}

\paragraph{\bf Digraphs.}
Given $K$ and $z$, we define a 
(vertex-weighted) directed graph with underlying graph $K$
as follows.
If $v$ and $w$ are adjacent vertices in $K$ and $z_v\ge z_w$,
there is an edge directed from $v$ to $w$,
Note that if $z_v=z_w$ then there are two directed edges
between $v$ and $w$ ($v\rightleftharpoons w$); we call this a
{\it two-sided  edge}.

For a typical $z$, the sources and sinks of the directed graph
are local maxima and minima of $z$.
In particular, a local minimum occurs at a source 
if and only if it has out-degree~0 and
a local maximum occurs at a source 
if and only if it has in-degree~0.
A vertex of out-degree 1 (and arbitrary in-degree)
plays no role in the persistence diagram; this applies in particular
to local maxima occurring at leaves of the graph.

For any typical $z$, there is a unique vertex 
$i_{\mathrm{min}}$ for which $z$ is a minimum; 
the corresponding point in the persistence diagram
$P_0(z)$ is $(z_{i_\mathrm{min}},\infty)$.  
All other persistence points are finite,
and each persistence point occurs exactly once.

\begin{lemma}
Suppose that $z$ is a typical point, and 
$\textrm{out-degree}(i)\ne1$. Then  there is 
a persistence point with coordinate $z_i$.
\end{lemma}

\begin{proof}
If $i$ has out-degree~0, then $z_i$ is a local minima,
and hence there is a persistence point $(p^b,p^d)$ with 
$p^b=z_i$.
If $i$ has out-degree~$\ge2$ then at least~2 components
of $K_r$ will merge at $r=z_i$, 
and thus there is a persistence point $(p^b,p^d)$ with 
$p^d=z_i$.
\end{proof}

\begin{subremark}
If $z_i$ is  not  a  typical  point then, after identifying
vertices connected by  2-sided edges,
a similar argument applies. We leave the verification to the reader.
\end{subremark}

We say that $z_i$ is a \emph{critical} coordinate if the number of out-edges (after the potential identification) from $i$ is not exactly one.

\section{Persistence diagrams on an $N$-gon}
\label{sec:components}

Throughout this section $K$ is an $N$-gon, the vertices 
are identified with elements of $\Z/N$ (with $0=N$) and there are $N$ edges, $[i,i+1]$, $0\le i<N$.
Our main result (Corollary \ref{cor:circle}) concerns  the preimage $\data_P$ in $\R^N$ of a typical persistence diagram: when $N/2$ is greater than the number $M$ of persistence points of $P$, the connected components of $\data_P$ are homotopy equivalent to $S^1$.

The following result is analogous  to \cite[Lemma 2.4]{CMW}.

\begin{lemma}\label{lem:extrema}
Let $z\in\R^N$ be a typical point with persistence diagram
$\PH_0(z)=\{ (p_m^b,p_m^d)\}_{m=1}^M$. Then $z$ has $2M$ local extrema;
each $p_m^d$ is a local maximum, and each $p_m^b$ is a local minimum.

Since there are only finitely many critical value
sequences with these local maxima and minima, 
only finitely many critical value sequences arise. 
\end{lemma}

\begin{subremark}
Conversely, if $PH_0(z)=\{ (p_m^b,p_m^d)\}_{m=1}^M$ is a 
persistence diagram with distinct values, 
there is a typical point $z'$ with $PH_0(z)=PH_0(z')$.
\end{subremark}

When $N=2M$, every $z_i$ is a local extremum, so $\data_P$ is a finite
set. We shall assume that $N>2M$ for the rest of the paper.

\begin{example}\label{ex:N=2M+1}
When $N=2M+1$, every point $(z_1,...,z_N)$ has exactly one 
non-extremal value,
every component of $\data_P$ is homeomorphic to a circle, and the 
components are indexed by the critical value sequences 
modulo cyclic rotations.

To see this, fix a critical value sequence $v=(a,b,...,s,t)$ using the $2M$ values 
$p_m^b,p_m^d$ in Lemma~\ref{lem:extrema};
there are only $N$ places to insert a non-extremal element.
Let $C_i(v)$, $i=0,\ldots,N-1$ denote the subspace of points in $\data_P$ 
with critical value sequence $v$, where $z_i$ is the non-extremal value.
Then $C_i(v)$ is an open interval, whose closure $\bar{C}_i(v)$ meets 
$\bar{C}_{i-1}(v)$ and $\bar{C}_{i+1}(v)$ in an endpoint when $i\ne0,N-1$. 
Writing $Rv$ for the cyclic rotation $(t,a,b,...,s)$ of $v$,
the closure of $C_0(v)$ meets $\bar{C}_1(v)$ and $\bar{C}_{N-1}(Rv)$, 
as $z_0$ is between $z_1=a$ and $z_{N-1}=d$.
Similarly, the closure of $C_{N-1}(v)$ meets $\bar{C}_{N-1}(v)$
and $\bar{C}_0(R^{-1}v)$.

To count the number of components, note that
the cyclic order of the local extrema cannot be changed without 
changing the persistence diagram, all points in $C(z)$ must have 
the same critical value sequence as $z$, up to rotation.
\end{example}





\medskip
\section{The poset $\Str$ of cellular circular strings}
\label{sec:Str}

Throughout this section $K$ is an $N$-gon.
A {\it circular symbol string} is a string of symbols
$s=s_1\cdots s_N$, where each symbol is either $0,1$ or $X$;
we will refer to 0 and 1 as {\it bits}. (Cf.\,\cite{CMW}.)
Any associated symbol string has a canonical representation as the concatenation $s=\gamma_1\cdots\gamma_J$ of {\it blocks} $\gamma_i$, each block consisting of the same symbol, such that adjacent blocks have
different symbols. 
Because of our wrap-around convention, it is possible that the last block has the same symbol as the first.

\begin{definition}\label{def:cc-string}
A circular symbol string $s=s_1\cdots s_N$ is a {\it circular cellular string} of {\it rank} $M>0$
if for the canonical representation $s=\gamma_1\cdots\gamma_J$:
\begin{enumerate}
    \item[(i)] exactly $M$ blocks have symbol 0, and 
    \item[(ii)] if $\gamma_j$ consists of the symbol $X$ and $j\ne1,J$, then
the symbols of $\gamma_{j-1}$ and $\gamma_{j+1}$ are different.
\end{enumerate}
The \emph{dimension} of a cellular string $s$, $\dim(s)$, is the
 number of symbols $X$ in $s$; it is at most $N-2M$.
\end{definition}

Fix $N$ and $M<N/2$. The set $\Str=\Str(N,M)$ 
of circular cellular strings of length $N$ and rank $M$
is a poset, where $s'<s$ if the string $s$ is obtained from $s'$ by
replacing some of the bits 0 and 1 in $s'$ by $X$.
For example, in $\Str(3,1)$ we have $X01 > 001 < 0X1 > 011 < 01X$.

\begin{proposition}
The maximal elements of $\Str(N,M)$ are the strings of dimension $N-2M$.
\end{proposition}

\begin{proof} (Cf.\,\cite[Prop.\,2.8]{CMW}.)
If $s$ has smaller dimension then there is a block of length $\ge2$
of symbols $0$ or $1$. Replacing the first symbol in the block by $X$
yields a symbol $s'$ with $s<s'$, so $s$ is not maximal.
\end{proof}

\begin{lemma}\label{lem:glb}
Every string $s'$ in $\Str$ is the greatest lower bound of the set of maximal strings $s$ with $s' < s$.

Every maximal chain in $\Str$ has length $N-2M$.
\end{lemma}

\begin{proof}
The proof of  \cite[Lemma 2.9]{CMW} goes through. Briefly,
we proceed by downward induction on the dimension $d=\dim(s')$. Replacing 
the two end symbols of a block by $X$ yields two $(d+1)$-dimensional
strings whose greatest lower bound is $s'$.
\end{proof}

\bigskip
\section{The polytopes for the $N$-gon}\label{sec:polytopes}

Throughout this section we work the $N$-gon.
Fix a critical value sequence $(z_{n_1},...,z_{n_{2M}})$ 
of a typical point $z$.
To each circular cellular string $s$, represented in block form as $\gamma_1\cdots\gamma_J$, we associate a polytope $T(s)=\prod T(\gamma_j)$ as in \cite{CMW}:
\begin{itemize}
    \item if $\gamma_j$ is the
$k^{th}$ block involving 0 or 1 we set $T(\gamma_j)=\{z_k\}^{n_k}$;
    \item if $\gamma_j$ involves $X$ and $\gamma_{j-1}$ is the $k^{th}$ block
involving 0 or 1, we define $T(\gamma_j)$ to be the simplex
of all monotone sequences $(x_1,...,x_{n_j})$ of length $n_j$
between $z_k$ and $z_{k+1}$.
\end{itemize}

Fix a persistence diagram $P$, and a component $C$ of $\data_P$.
Then for any typical point $z$ in $C$, it is clear from the definition
\ref{def:cvs} of a critical value sequence that $C$ is the union
of the simplices $T(s)$, where $s\in\Str(N,M)$
For this, it is convenient to work with the poset of
circular cellular strings.

Let $s$ be a circular cellular string with $k$ blocks with symbols $X$,
of lengths $n_1,...,n_k$.
Recall from \cite[Example 2.10, Theorem 2.13]{CMW} that the geometric realization of
the sub-poset $\Str/s = \{s': s'\le s\}$ is homeomorphic to the
product $\Delta^{n_1}\times\cdots \Delta^{n_k}$ of simplices,
i.e., to $T(s)$.

\begin{theorem}\label{thm A}
If $z$ is a typical point, the connected component $C(z)$ 
of $\data_P$ is homeomorphic to the geometric realization of $\Str$.
\end{theorem}

\begin{proof}
The proof in \cite[Theorem 2.13]{CMW} 
goes through. The key observation
is that for each $s_1,...,s_n$, the intersection of the realizations 
of the $\Str/s_i$ is the realization
of $\Str/s'$, where $s'$ is the greatest lower bound of the $s_i$.
\end{proof}

\begin{example}
In \cite{LT}, a similar problem is studied with a different filtration,
intermediate between the polygon with $N$ vertices and its subdivision,
which has $2N$ vertices. The comparison is sketched in
Section 5.2 of \cite{LT}.
\end{example}


\goodbreak 
\section{Homotopy operations}\label{sec:homotopy}

Let $\Str_0$ (resp., $\Str_1$) denote the sub-poset of 
strings in $\Str$ whose initial bit is 0 (resp., 1), such that
$s_1$ and $s_N$ are not both 0 (resp., 1). For example,
$0X1X$, $X0101$ and $XXX01$ are in $\Str_0$.

\begin{proposition}
The classifying spaces of $\Str_0$ and $\Str_1$ are contractible.
\end{proposition}

\begin{proof}
This is the content of Proposition 3.5 and Corollary 3.6 in \cite{CMW}.
The sub-poset of $\Str$ in {\it loc.\,cit.}\ consisting
of strings such that $s_1$ and $s_N$ are not both 0
is our $\Str_0$, and the
poset morphisms used in that proof send $\Str_0$ to itself.
The realizations of those poset morphisms, when composed,
give a homotopy from $B\Str_0$ to a point. 
The proof for $\Str_1$ is the same.
\end{proof}

For symbols $a,b$ we write $\Str_{ab}$ for the sub-poset of
strings in $\Str$ whose initial and terminal symbols are 
$a$ and $b$, respectively; we abbreviate such a string as 
$a\sigma b$, where $\sigma$ is a string of length $N-2$.


Let $\overline\Str_{00}$ 
denote the sub-poset of strings in $\Str$ whose initial and terminal 
symbols are either: both 0; 0 and $X$; or $X$ and 0.
Thus $\overline\Str_{00}$
contains $\Str_{00}$ as well as $\Str_{0X}$ and $\Str_{X0}$
and is disjoint from $\Str_0$.
Since the initial bit for $s\in\Str_{X0}$ is 1,
\begin{equation}\label{eq:00}
\overline\Str_{00} \cap \Str_0 = \Str_{0X}, \quad
\overline\Str_{00} \cap \Str_1 = \Str_{X0}.
\end{equation}
We define $\overline\Str_{11}$ similarly, by interchanging 0 and 1.
Thus: 
\begin{equation}\label{eq:11}
\overline\Str_{11} \cap \Str_1 = \Str_{1X}, \quad
\overline\Str_{11} \cap \Str_0 = \Str_{X1}.
\end{equation}

\begin{lemma}\label{retract:00}
$B\Str_{00}$ is a deformation retract of $B\overline\Str_{00}$.
\\
By symmetry, $B\Str_{11}$ is a deformation retract of $B\overline\Str_{11}$.
\end{lemma}

\begin{proof}
Define $R:\overline\Str_{00}\to\Str_{00}$ to be the identity on
$\Str_{00}$, and $R(0\sigma X)=0\sigma0$, $R(X\sigma0)=0\sigma0$.
It is easy to see that $R$ is a poset map, and that $R(s)\le s$,
i.e., $R\Rightarrow\mathrm{id}$ is a natural transformation.
Taking the geometric realization, we see that $R$ is a continuous map,
and that $R$ is homotopic to the identity on $B\overline\Str_{00}$.
\end{proof}


\begin{definition}\label{def:F1}
If $s$ is a circular cellular string in $\Str_{00}$, 
we define $F_1(s)$ as follows (cf.\ \cite[Def.\,3.1]{CMW}).
If $s=0X\sigma0$, set $F_1(s)=s$. If not, there are two cases.
Case (i): if $s$ has no 00 or 11 preceding the 
leftmost $X$, $F_1(s)$ transposes that $X$ with the bit immediately preceding it.  
Case (ii): if $s$ has the form $\sigma_1abb\sigma_2$,
where $\sigma_1a$ is an alternating bitstring (beginning with 0)
and $\sigma_2$ is the remainder of the string,
we set $F_1(s)=\sigma_1aab\sigma_2$. Note that in case (ii),
$\sigma_2$ is either empty or ends in 0.

Let $\Str_{00}^{(\ell)}$ denote the sub-poset of $\Str_{00}$
consisting in strings beginning $0X\!\cdots\!X$ ($\ell\!-\!1$ symbols $X$).
The definition of $F_\ell:\Str_{00}^{(\ell)}\to\Str_{00}^{(\ell)}$ 
mimicks that of $F_1$; 
if $s=0\beta\sigma$ (where $\beta$ is a sequence of
$(\ell\!-\!1)$ symbols $X$) then $F_\ell(s)=0\beta F_1(\sigma)$.
\end{definition}

\begin{lemma}
$F_1:\Str_{00}\to\Str_{00}$ is a poset morphism and 
$F_1^k(\Str_{00})=\Str_{00}^{(2)}$ for $k\gg0$.
\end{lemma}

\begin{proof}
We proceed by downward induction to show that if $s'<s$
then $F_1(s')\le F_1(s)$. If the initial $X$ in $s'$ is not preceded 
by a $00$ or $11$, the same is true for $s$, and the inequality is evident.
Next, suppose that $s'= \sigma_1 abb...b \sigma_2$,
where $\sigma_1a$ is an alternating bitstring.
If $s= \sigma_1 abb...b \sigma_2'$ with $\sigma_2 < \sigma_2'$,
we also have $F_1(s')<F_1(s)$. Otherwise,
either $s\ge s_1$ or $s\ge s_2$, where 
$s_1=\sigma_1 aXb\cdots b \sigma_2$ and $s_2=\sigma_1 ab\cdots X \sigma_2$.
Since $F_1(s')<F_1(s_1)$ and $F_1(s')< F_1(s_2)$, 
the result follows by induction.
\end{proof}

\goodbreak
\begin{proposition}\label{p:00}
$\Str_{00}$, $\overline\Str_{00}$, $\Str_{11}$ and 
$\overline\Str_{11}$ are contractible.
\end{proposition}

\begin{proof}
We give the proof for $\Str_{00}$; it follows by
symmetry and Lemma \ref{retract:00} that 
$\overline\Str_{00}$, $\Str_{11}$ and $\overline\Str_{11}$ 
are also contractible.

We first show that $\Str_{00}^{(2)}\to\Str_{00}$ is a homotopy equivalence.
As in \cite[Proposition 3.5]{CMW}, we filter $\Str_{00}$ by sub-posets $Fil_i$,
where $Fil_0=\Str_{00}^{(2)}$ and $Fil_i$ is the full sub-poset 
on the strings $s$ with $F_1^i(s)\in\Str_{00}^{(2)}$. Since
$F_1$ maps $Fil_i$ to $Fil_{i-1}$, the geometric realization
of $F_1$ restricts to a continuous map from $BFil_i$ to $BFil_{i-1}$.

To see that $BFil_{i-1}\subseteq BFil_i$ is a homotopy equivalence,
we define a poset endomorphism $h$ on $Fil_i$ as follows.
If $s\in Fil_{i-1}$ then $h(s)=s$. Otherwise, define $h(s)$
to be the greatest lower bound of $s$ and $F_1(s)$. Thus
$Bh$ is a retract of $BFil_i$ onto $BFil_{i-1}$.
For $s\in Fil_i$, the inequalities $F_1(s)\le h(s)\ge s$
yield natural transformations from $h$ to $F_1$ and to the identity,
and hence homotopies between the identity map, $Bh$, and $BF_1$.
These homotopies show that $BFil_{i-1}\simeq BFil_i$.
Composing these homotopies gives a homotopy equivalence
between $B\Str_{00}$ and $BFil_0=B\Str_{00}^{(2)}$.

The same argument, {\it mutatis mutandis}, shows that each inclusion 
$\Str_{00}^{(\ell-1)}\to\Str_{00}^{(\ell)}$ is a homotopy equivalence.
Since $B\Str_{00}^{(\ell)}$ is the point $\{0X\cdots X10\cdots10\}$
when $\ell=N-2m+1$, 
$B\Str_{00}$ is homotopy equivalent to a point, as claimed.
\end{proof}

\begin{proposition}\label{p:0X}
$\Str_{0X}$ and $\Str_{X0}$ are contractible.
\\
By symmetry, $\Str_{1X}$ and $\Str_{X1}$ are also contractible.
\end{proposition}

\goodbreak
\begin{proof}
Since the posets $\Str_{0X}$ and $\Str_{X0}$ are isomorphic 
(by the front-to-back permutation of strings), it suffices
to give the proof for $\Str_{0X}$. By Example \ref{ex:N=2M+1},
We may assume that $N>2M+1$.
Definition \ref{def:F1} goes through word for word
in this setting to yield a poset endomorphism $F_1$ on $\Str_{0X}$,
with the image of $F_1^K$ being $\Str_{0X}^{(2)}$ for $K\gg0$.
Now the proof of Proposition \ref{p:00} goes through to show
that $\Str_{0X}$ is contractible.
\end{proof}

\section{Circular components}\label{sec:main}

Let $Q$ denote the 8-element poset on the left of diagram \eqref{lattice};
the 4 corners are minimal elements, and the 4 side-vertices are maximal.
The geometric realization $BQ$ of $Q$ has a vertex for each element of $Q$
and an edge for each strict inequality; there are no higher simplices
because the poset $Q$ has no chains $q_0<q_1<q_2$. Thus $BQ$
is an octagon, homeomorphic to a circle.
The sub-posets of $\Str$ we have described fit into 
the right-hand diagram below,
where the arrows indicate inclusion.
\begin{equation}\label{lattice}
\xymatrix@R=1.0em{
0X \ar[r]\ar[d] & 0 & X1 \ar[l]\ar[d] \\
00 & Q & 11 \\
X0 \ar[r]\ar[u] & 1 & X1 \ar[l]\ar[u]
}\quad
\xymatrix@R=1.0em{
\Str_{0X} \ar[r]\ar[d] & \Str_0 & \Str_{X1} \ar[l]\ar[d] \\
\overline\Str_{00} && \overline\Str_{11} \\
\Str_{X0} \ar[r]\ar[u] & \Str_1 & \Str_{1X} \ar[l]\ar[u]
}\end{equation}

Define $f:\Str\to Q$ by sending elements of $\Str_{0X}$, $\Str_{0X}$, 
$\Str_{1X}$ and  $\Str_{X1}$ to the corresponding 
minimal vertices of $Q$, as indicated by \eqref{lattice};
strings in $\Str_{00}$ and $\Str_{11}$ are sent to the
vertices indicated by 
$\overline\Str_{00}$ and $\overline\Str_{11}$,
respectively. The strings in $\Str_0$ not in $\Str_{0X}$ or $\Str_{X1}$
(resp., in $\Str_1$ not in $\Str_{X0}$ or $\Str_{1X}$)
are sent to the other maximal vertices of $Q$, as indicated.
It is clear that $f$ is a poset morphism.

Recall from \cite[IV.3.2.3]{WK} that for $q\in Q$, 
the comma category $f/q$ has objects the pairs $(s,q)$,
where $f(s)\le q$, i.e., $s\in\Str$ such that $f(s)\le q$; 
and there is a morphism from $(s',q)$ to $(s,q)$ 
if and only if $s'\le s$ in $\Str$. 

\begin{lemma}
The right-hand side of diagram \eqref{lattice} is the diagram
of the comma categories $f/q$ for $q\in Q$.
\end{lemma}

\begin{proof}\label{lem:f/q}
For the minimal elements $q=ab$ of $Q$, it is a tautology that
$f/q = f^{-1}(q) = \Str_{q}$. Since $\overline\Str_{00}$ is
the union of $f^{-1}(00)$, $f^{-1}(0X)$ and $f^{-1}(X1)$,
we see that $f/0=\overline\Str_{00}$; by symmetry we also have
$f/1=\overline\Str_{11}$. The definition of $f$ on $\Str_0$
and $\Str_1$ ensures that we also have
$f/0=\Str_0$ and $f/1=\Str_1$.
\end{proof}

\goodbreak
\begin{theorem}\label{thm B}
$B\Str$ is homotopic to the circle $S^1$.
\end{theorem}

\begin{proof}
Quillen's Theorem~A says that if the geometric realization of
every $f/q$ is contractible, then $Bf: B\Str\to BQ \simeq S^1$ 
is a homotopy equivalence (see \cite[IV.3.7]{WK}).
By \eqref{eq:00}, \eqref{eq:11}, Propositions \ref{p:00} and \ref{p:0X}, 
the geometric realizations of all the $f/q$ are contractible.
\end{proof}

Combining Lemma \ref{lem:tpz'} with
Theorems \ref{thm A} and \ref{thm B}, we obtain:

\begin{corollary}\label{cor:circle}
If $P$ is a typical persistence diagram,
every connected component
of $\data_P$ is homotopy equivalent to $S^1$.
\end{corollary}

\goodbreak 
\section{$Y$-shaped configurations}\label{sec:Y} 

In this section we show that $\data_P$ can still be homotopic to a circle, 
even for rooted trees with three branches 
and persistence diagram $P=\{ (0,\infty), (1,4) \}$ for $H_0$.
(The choice of $0<1<4$ is for concreteness.)


For simplicity, we focus on the case where the tree $K$ has 
vertex set $V = \{ i\mid i=1,\ldots, 7 \}$  and edges 
\[
E=\{ [1,2],[2,3],[3,4],[4,5],[3,6],[6,7] \}.
\]
That is, the central vertex $3$ has degree 3, and the 
endpoints are vertices $1$, $5$, and $7$.
The three branches ($\alpha$, $\beta$, and $\gamma$) are generated by the vertices $\{1,2,3\}$, $\{3,4,5\}$, and $\{3,6,7\}$, respectively.

Figure \eqref{eq:K-turn} illustrates six points 
$z_A, z_B, ... $ in $\data_P$ with their directed graphs $\Gamma(z_A), \Gamma(z_B), \ldots$ (as in Section \ref{sec:prel}).
The critical coordinates are marked by their values $0$, $1$, and $4$.
The other $z_i$ are marked by $X$; they not critical coordinates and their exact value is unimportant.

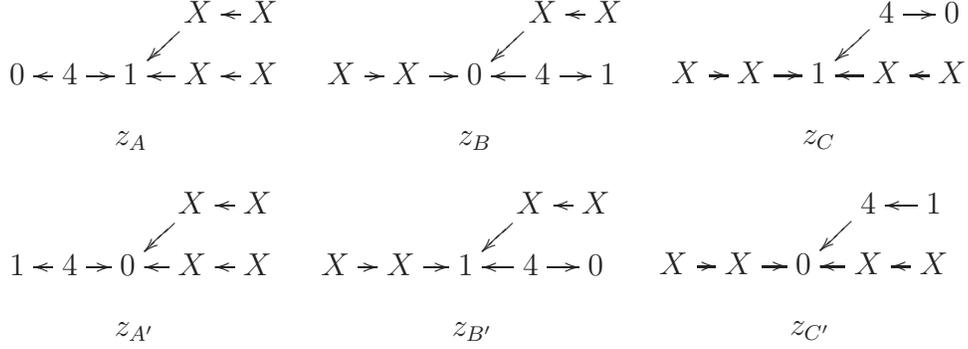
\begin{equation}\label{eq:K-turn}
\xymatrix@R=0.7em@C=7.0pt{ 
  &&               & X\ar[dl] & X \ar[l] \\ 
0 & 4 \ar[l]\ar[r]& 1 & X \ar[l] & X \ar[l] \\ 
&& z_A
}\quad
\xymatrix@R=0.7em@C=7.0pt{
  &&               & X\ar[dl] & X \ar[l] \\
X \ar[r] & X \ar[r]& 0 & 4 \ar[l]\ar[r] & 1  \\
&& z_B
}\quad
\xymatrix@R=0.7em@C=7.0pt{
  &&               & 4\ar[r]\ar[dl] & 0  \\
X \ar[r] & X \ar[r]& 1 & X \ar[l] & X \ar[l] \\ 
&& z_C                       \\         
}
\end{equation}
\begin{equation*} 
\xymatrix@R=0.7em@C=7.0pt{  
  &&               & X\ar[dl] & X \ar[l] \\
1 & 4 \ar[l]\ar[r]& 0 & X \ar[l] & X \ar[l] \\ 
&& z_{A'} \hspace{-5pt} 
}\quad
\xymatrix@R=0.7em@C=7.0pt{
  &&               & X\ar[dl] & X \ar[l] \\
X \ar[r] & X \ar[r]& 1 & 4 \ar[l] \ar[r] & 0  \\ 
&&z_{B'} \hspace{-5pt} 
}\quad
\xymatrix@R=0.7em@C=7.0pt{
  &&               & 4\ar[dl] & 1 \ar[l] \\
X \ar[r] & X \ar[r]& 0 & X \ar[l] & X \ar[l] \\ 
&&z_{C'} \hspace{-5pt}                      \\
}\end{equation*}


\begin{construction}\label{K-turn}
Consider the point $z_A\in\R^7$, illustrated on the left of \eqref{eq:K-turn}.
The extremal points $0$, $4$ and $1$ all lie on branch $\alpha$.
In analogy with \cite{CMW}, it is possible to slide the $041$ two places to the right (to $z_B$ on branch $\beta$) without  
changing the directed edges on branch $\gamma$. 
Beginning with $z_B$, we can slide $041$ clockwise up 
(to $z_C$ on branch $\gamma$) without 
changing the directed edges on branch $\alpha$.
We now slide 140 two places to the left 
(to $z_{A'}$ on branch $\alpha$) without changing the directed edges on branch $\beta$. 

This maneuver (which resembles a clockwise 'K-turn' in a car)
results in the critical coordinates $0$ and $1$ switching places.
Following this with a second K-turn ($z_{A'}$ to $z_{B'}$ to $z_{C'}$ to $z_A$)
returns us to the starting configuration.
\end{construction}

Construction \ref{K-turn} shows that the six points in \eqref{eq:K-turn} lie on a non-trivial loop;  it will turn out to be a generator of the fundamental group of $\data_P$.
While obvious, it is perhaps worth emphasizing that these sequences of slides are not possible on the interval or $N$-gon.
These new moves make it possible for the topology of $\data_P$ to be more interesting.

\medskip
Let $\branchA$ be the subspace of $\data_P$ consisting of all $z$ with $z_i=0$ for some vertex $i$ on branch $\alpha$, and $z_j\ne4$ for all vertices $j$ on branch $\alpha$ further from the central vertex than vertex $i$.
In Figure \eqref{eq:K-turn}, $z_A$, $z_B$ and $z_{C'}$ are in $\branchA$.

The subspaces $\branchB$ and $\branchC$ are defined similarly.

\begin{lemma}\label{lem:ABC} 
$\branchA$, $\branchB$ and $\branchC$ are contractible. 
\end{lemma}

\noindent
Figure \eqref{eq:Br-A} illustrates the steps in the proof,
starting from $z_{C'}$.

\begin{proof}
By symmetry, it suffices to consider $\branchA$.
We shall use three steps to construct a deformation retraction of $\branchA$ to a point.  
By definition there exists a vertex $i_0$ closest to the endpoint $i_\alpha=1$ of branch $\alpha$ with $z_{i_0}=0$.
For the first step, continuously decrease the value of $z_j$ to $0$ for every vertex between $i_\alpha$ and $i_0$ (and do nothing if $i_0=i_\alpha$).
This is a deformation retraction onto the subspace of all $z'$ which are 0 at $i_\alpha$. 

For the second step of the homotopy, given a point with $z_{i_\alpha}=0$, consider the set of vertices $j$ 
such that the path from vertex $i_\alpha$ to vertex $j$ does not contain 
a vertex $k$ with $z_k=1$. We can continuously change the values of $z_j$ at all
these vertices to $4$. This is a deformation retraction to the subspace
of all $z''$ where $z''=4$ at vertex 2 
(the vertex adjacent to $i_\alpha=1$).

Finally, given $z''$, we can continuously decrease the value of $z''_j$ to $1$ for all vertices $j$ other than 
$i_\alpha$ and its neighbor.
The result is deformation retraction to the point $z=(0,4,1,...,1)$, 
showing that $\branchA$ is contractible.
\end{proof}

\begin{equation}\label{eq:Br-A}
\xymatrix@R=0.7em@C=7.0pt{
  &&               & 4\ar[dl]\ar[r] & 1  \\
0 & 0 \ar[l]\ar[r]& 0 & X \ar[l] & X \ar[l] \\ 
}\quad
\xymatrix@R=0.7em@C=7.0pt{
  &&               & 4\ar[dl]\ar[r] & 1 \\
0 & 4 \ar[l]\ar[r]& 4 & 4 \ar[l] & 4 \ar[l] \\
}\quad
\xymatrix@R=0.7em@C=7.0pt{
  &&               & 1\ar[dl] & 1 \ar[l] \\
0 & 4 \ar[l]\ar[r]& 1 & 1 \ar[l] & 1 \ar[l] \\
}
\end{equation}
\smallskip

We write $\branchAC$ for $\branchA\cap\branchC$.
It is the subspace of all $z$ where $z=0$ at the central vertex, 
while the vertex with $z=1$
(and hence the vertex with $z=4$) lies on branch $\beta$.
For example, $z_B$ is in $\branchAC$.
The subspaces $\branchAB$ and $\branchBC$ are defined similarly.

\begin{lemma}\label{lem:AC}
$\branchAC$, $\branchAB$ and $\branchBC$ are contractible.
\end{lemma}

\begin{proof}
By symmetry, it suffices to consider $\branchAC$.
For the 7-vertex tree, $\branchAC$ consists of just the
points of the form $z_B$, 
illustrated by the second diagram of \eqref{eq:K-turn}.
In particular, $\branchAC$ is contractible.
\end{proof}

\begin{subremark}\label{rem:longer}
The proofs of Lemmas \ref{lem:ABC} and \ref{lem:AC}
go through for
longer $Y$-shaped trees, i.e., rooted trees with a
central vertex of degree 3 with 3 linear branches of
length $\ge2$ attached to it. (The vertices $1,2,3$ are 
at the end of branch $\branchA$.)
\end{subremark}


By inspection, every point in the preimage $\data_P$
lies in one of the subspaces $\branch_q$. Since 
the intersection of any two branches is contractible,
we see that the preimage $\data_P$ is path-connected.

Let $Q$ denote the 6-element poset on the left of diagram
\eqref{hexagon}; the elements $\alpha,\beta,\gamma$ are maximal 
and the others are minimal. Thus $BQ$ is a hexagon, 
homeomorphic to the circle $S^1$.

Consider the topological poset $\branch$ of pairs $(x,q)$ with
$x\in \branch_q$, illustrated by  the right of \eqref{hexagon}.
It is clear that there is a poset morphism
$f:\branch\to Q$ sending elements $(x,q)$ to $q$.
\begin{equation}\label{hexagon}
\xymatrix@R=1.0em{
\alpha  \ar[r]\ar[d]  & \alpha\gamma   & \gamma \ar[l] \ar[d] \\
\alpha\beta  &\beta  \ar[l]\ar[r] & \beta\gamma\\
}\quad
\xymatrix@R=1.0em{
\branchA  \ar[r]\ar[d]  & \branchAC  & \branchC  \ar[l]\ar[d]\\
\branchAB & \branchB  \ar[r]\ar[l]&\branchBC \\
}\end{equation}

\begin{lemma}\label{|Br|}
The geometric realization of $\branch$ is homotopy equivalent to $\data_P$.
\end{lemma}

\begin{proof}
For each $q$, the realization $|\branch|$
contains a subspace homeomorphic to
$\branch_q$, and for each $q'<q$ the realization contains the mapping cylinder of the inclusion $\branch_{q'}\subset\branch_q$.
Thus there is a natural map from $|\branch|$ 
onto $\data_P$.
Since $\branch_{q'}$ is a subspace of two subspaces $\branch_{q}$, 
it is easy to see that $|\branch|\to \data_P$
is a homotopy equivalence.
\end{proof}

\begin{theorem}\label{thm:K-turn}
The preimage $\data_P$ is homotopic to $S^1$.
\end{theorem}

\begin{proof}
By Lemma \ref{|Br|}, it suffices to show that $f$ induces
a homotopy equivalence $|\branch|\to BQ\simeq S^1$.
By Quillen's Theorem~A, it suffices to show that (the realization of)
each comma category $f/q$ is contractible.

The comma category $f/q$ is the poset of all pairs
$(x,q'\le q)$ with $x\in\branch_{q'}$. If $q$ is minimal in $Q$,
$f/q=f^{-1}(q)=\branch_q$, which is contractible by 
Lemma \ref{lem:AC}. 
If $q$ is maximal, we still have $f^{-1}(q)=\branch_q$ 
but $f/q$ contains elements $(x,q'<q)$.
The geometric realization of the natural transformation 
$\eta:(x,q'\le q)\Rightarrow(x,q)$ is a deformation retraction
from $B(f/q)$ to the subspace $B(f^{-1}q)$,
which is contractible by Lemma \ref{lem:ABC}.
\end{proof}



\section{Star-like configurations}\label{sec:starlike}

In this section, we generalize from $Y$-shaped trees to star-like trees,
i.e., trees with a central vertex of degree $n$ and $n$ linear branches
of length at least~2 attached to it. 
(The $Y$-shaped trees of Section \ref{sec:Y} form the case $n=3$.)

For $q=1,...,n$, let $\branch_q$ be the subspace of all $z$ with $z_i=0$ for some vertex on branch $q$, and $z_j\ne4$ for all vertices $j$ on branch $q$ further from the central vertex than $i$.
The proof of Lemma \ref{lem:ABC} goes through to show that $\branch_q$ is contractible for each $q=1,...,n$.

Set $\branch'_q = \bigcap_{p\ne q}\branch_p$;
it is the subspace of all $z$ where $z=0$ at the central vertex, 
while there is a vertex with $z=1$
(and hence a vertex with $z=4$) lies on branch $q$.
The proof of Lemma \ref{lem:AC} goes through to
show that each $\branch'_{q}$ is contractible.

Since each point of $\data_P$ lies on one of the
branches, which are contractible,
and each $\branch'_q$ is contractible, 
$\data_P$ is path-connected.

\begin{theorem}\label{thm:starlike}
$\data_P$ is homotopy equivalent to a bouquet 
$\bigvee S^1$ of $(n^2-3n+1)$ circles.
\end{theorem}

When $n=3$, this yields 1 circle, as in Theorem \ref{thm:K-turn};
for $n=4$ branches, $\data_P$ is homotopy
equivalent to a bouquet of 5 circles.

\begin{proof}
Consider the poset $Q$ whose elements
are the $2n$ branches $\branch_q$ and $\branch'_q$,
with $\branch'_q < \branch_p$ for every $p\ne q$.
The realization of this poset is a bipartite graph $\Gamma$ such that every vertex of $\Gamma$ has degree $n-1$. 
Since $\Gamma$ has $2n$ vertices and $n^2-n$ edges, its Euler characteristic is 
$$\chi=V-E=3n-n^2.$$
Since $\Gamma$ is connected, and 
$\chi=\dim H_0(\Gamma)-\dim H_1(\Gamma)$,
$\Gamma$ is homotopy
equivalent to a bouquet of 
$1-\chi= (n^2-3n+1)$ circles.

Consider the topological poset $\branch$ of pairs
$(x,s),$ where $x\in \branch_s$, and pairs
$(x,s'),$ where $x\in \branch'_s$;
there is an obvious poset morphism $f:\branch\to Q$,
and hence a map $|\branch|\to\Gamma$.
The proof of Lemma \ref{|Br|} goes through
(with 'two' replaced by $n-1$) to show that
$|\branch|$ is homotopy equivalent to $\data_P$.
Finally, the proof of Theorem \ref{thm:K-turn}
goes through to show that 
$|\branch|\to \Gamma$ is a homotopy equivalence.
(One uses the version of Quillen's Theorem~A for
the realization of topological categories;
see \cite[IV.3.9]{WK}.)
The homotopy equivalence of the theorem follows.
\end{proof}


\end{document}